\documentclass[12pt, reqno]{amsart}
\usepackage{amsmath, amsthm, amscd, amsfonts, amssymb, graphicx, xcolor, enumerate}
\usepackage{multicol, lipsum, paracol, diagbox, tabularx, multirow, makecell}
\usepackage[mathcal]{euscript}

\usepackage{dsfont, booktabs}

\usepackage[bookmarksnumbered, colorlinks, plainpages]{hyperref}

\theoremstyle{definition}
\newtheorem{theorem}{Theorem}[section]
\newtheorem{lemma}[theorem]{Lemma}
\newtheorem{proposition}[theorem]{Proposition}

\theoremstyle{remark}
\newtheorem{remark}[theorem]{Remark}
\numberwithin{equation}{section}

\newcommand{\BC}{\mathbb{C}}

\newcommand{\BF}{\mathbb{F}}

\newcommand{\BK}{\mathbb{K}}

\newcommand{\BZ}{\mathbb{Z}}

\newcommand{\CF}{\mathcal{F}}

\newcommand{\CR}{\mathcal{R}}
\newcommand{\CS}{\mathcal{S}}

\newcommand{\gs}{\geqslant}
\newcommand{\ls}{\leqslant}
\newcommand{\pt}[1]{\left(#1\right)}
\newcommand{\cb}[1]{\left\{#1\right\}}
\newcommand{\jk}[1]{\left\langle#1\right\rangle}
\newcommand{\fk}[1]{\left[#1\right]}
\newcommand{\abs}[1]{\left|#1\right|}

\newcommand{\fS}{\mathfrak{S}}

\newcommand{\GL}{\mathrm{GL}}
\newcommand{\SL}{\mathrm{SL}}
\newcommand{\Cay}{\mathrm{Cay}}
\newcommand{\ord}{\mathrm{o}}
\newcommand{\jz}[4]{{\begin{pmatrix}
		#1&#2\\
		#3&#4
\end{pmatrix}}}
\newcommand{\one}{\mathds{1}}

\newcommand{\spec}{\mathrm{spec}}
\newcommand{\supp}{\mathrm{supp}}

\begin{document}
	\setcounter{page}{1}

	\centerline{}
	
	\centerline{}

	
\title[]{\large EKR R\lowercase{esults} \lowercase{for} S\lowercase{ome}  C\lowercase{uspidal} R\lowercase{epresentations}}
	
\author[]{\large J\lowercase{iaqi} L\lowercase{iao}}
	
\address{School of Mathematics and Physics, University of Science and Technology Beijing, Beijing 100083, China}

\email{\textcolor[rgb]{0.00,0.00,0.84}{ liaojq19@ustb.edu.cn}}
	
	
\thanks{}

\begin{abstract}

Let $\mathcal{S}_\chi$ be a cuspidal representation of $\mathrm{GL}_2(q)$. A subset $S \subseteq \mathrm{GL}_2(q)$ is $t$-intersecting if $\dim_{\mathbb{C}}\{v \in \mathcal{S}_\chi: g_1v = g_2v\} \geqslant t$ for any $g_1, g_2 \in S$. For $q = p^f$, where $p$ is an odd prime, $f \geqslant 2$, and $1 \leqslant t < p$ is an integer, we determine the maximum size of an $t$-intersecting subset when $\chi(-1) = -1$, and when $\chi(-1) = 1$ with $t \leqslant 2$. The proof combines explicit fixed-space dimensions, attaining subgroup constructions, and weighted spectral bounds for normal Cayley graphs.

\end{abstract} \maketitle
	
\section{Introduction}

Intersection problems ask how large a family of mathematical objects can be when every pair satisfies a prescribed intersection condition, and which families attain the maximum. A fundamental example is the Erd\H{o}s-Ko-Rado theorem \cite{MR140419}: if $n \gs 2k$ and $\CF \subseteq \binom{[n]}{k}$ is intersecting, then $\abs{\CF} \ls \binom{n - 1}{k - 1}$. The bound is attained by a star, consisting of all $k$-subsets containing a fixed element; when $n > 2k$, these are the only extremal families. This theorem initiated a broad study of intersection conditions for sets and other combinatorial structures, including vector spaces, graphs, and permutations. The surveys of Frankl and Tokushige \cite{MR3534067} and Ellis \cite{MR4421401} describe many of these developments, while Godsil and Meagher \cite{MR3497070} provide a systematic account of the algebraic methods used in the subject.

For subsets of group elements, a natural notion of intersection is agreement in a given action. Let $G \curvearrowright X$ be an action of a finite group $G$ on a finite set $X$. A subset $S \subseteq G$ is called $t$-intersecting with respect to this action if $\abs{\cb{x \in X: g_1x = g_2x}} \gs t$ for any $g_1, g_2 \in S$. For the natural action of the symmetric group $\fS_n$, Ellis, Friedgut and Pilpel \cite{MR2784326} proved that, for each fixed $t$ and all sufficiently large $n$, $t$-intersecting family $S$ satisfies $\abs{S} \ls (n - t)!$, with equality precisely when the family is the coset of the point-wise stabilizer of $t$ distinct points. Their work resolved the longstanding Deza-Frankl conjecture \cite{MR439648}.

Agreement also gives rise to intersection problems for linear maps. For the natural action of $\GL_n(q) \curvearrowright \BF_q^n$, the condition is $\dim_{\BF_q}\cb{v \in \BF_q^n: g_1v = g_2v} \gs t$. Results for the case $n = 2$ were obtained by Meagher and Razafimahatratra \cite{MR4521814}. More specifically, \cite[Theorem 1.2 and Section 2]{MR4521814} treats natural actions on non-zero finite-field vectors and already uses weighted derangement graphs and Hoffman's bound. Our intersection relation instead comes from a complex cuspidal representation. 

The present paper studies agreement in  complex irreducible representations of finite groups. Let $G \curvearrowright V$ be such a representation. We say that $S \subseteq G$ is $t$-intersecting with respect to this representation if $\dim_{\BC}\cb{v \in V: g_1v = g_2v} \gs t$ for any $g_1, g_2 \in S$. Writing $V^g := \cb{v \in V: gv = v}$, we have $\cb{v \in V: g_1v = g_2v} = V^{g_2^{-1}g_1}$. Thus $t$-intersection means $\dim_{\BC}V^{g_2^{-1}g_1} \gs t$ for any $g_1, g_2 \in S$. This leads to the following extremal question: for a given complex irreducible representation, what is the maximum possible size $\abs{S}$? We answer this question for the following parameter regimes of the cuspidal representations of $\GL_2(q)$.

\begin{theorem}\label{main}
	
Let $p$ be an odd prime, $f \gs 2$ and $q = p^f$, and $t \in \BZ$ with $1 \ls t < p$. Let $\BF$ be a field with $q$ elements, $\BK/\BF$ be quadratic and $\BK^1 := \cb{z \in \BK^*: z^{q + 1} = 1}$. Let $\chi$ be a complex multiplicative character of $\BK^*$ such that $\abs{\chi(\BK^1)} \ne 1$. Set $\ell := \abs{\chi(\BF^*)}$ and $k := \min\cb{k' \gs t: k'\mid q - 1}$. Let $\CS_\chi$ be the complex irreducible cuspidal representation of $\GL_2(q)$ associated with $\chi$. With respect to $\GL_2(q) \curvearrowright \CS_\chi$, if $S \subseteq \GL_2(q)$ satisfies that $\dim_\BC\cb{v \in \CS_\chi: g_1v = g_2v} \gs t$ for any $g_1, g_2 \in S$, then
\[\max\abs{S} = \left\{\begin{aligned}
	&\frac{\abs{\GL_2(q)}}{\ell},&\quad&\text{if } 2 \nmid \ell \text{ and } t = 1;\\
	&\frac{\abs{\GL_2(q)}}{2\ell},&\quad&\text{if } \chi(-1) = 1, \text{ and either } t = 2 \text{ or } (t = 1 \text{ and } 2\mid\ell);\\
	&\frac{\abs{\GL_2(q)}}{k\ell(q + 1)},&\quad&\text{if } \chi(-1) = -1.
\end{aligned}\right.\]

\end{theorem}

\begin{remark}
	
The condition $\chi(-1) = -1$ forces $\ell$ to be even and $m = (q - 1)/\ell = \abs{\ker\chi\cap\BF^*}$ to be odd. These parity conditions are used in the non-split support test and the character sums underlying Section \ref{case4}. When $\chi(-1) = 1$ and $t \ls 2$, the separate constructions in Sections \ref{case1}--\ref{case3} apply. The cases $\chi(-1) = 1$, $t \gs 3$ has neither of these available certificates and remains unsolved here. The assumptions $f \gs 2$ and $t < p$ also give $q/p - 1 \gs t$ which is needed for the non-semisimple classes in Lemmas \ref{eta} and \ref{Borel}. The prime-field case and larger thresholds require separate analysis.	
	
\end{remark}

The paper is organized as follows. Section \ref{Background} provides the necessary background in representation theory and spectral graph theory. Section \ref{Preliminaries} establishes several auxiliary lemmas involving dimension calculations, the construction of candidate maximum $t$-intersecting families, and the evaluation of useful character sums. Finally, in Section \ref{Proof}, we construct suitable weight functions and compute the eigenvalues of the associated weighted adjacency matrices to establish upper bounds on the sizes of $t$-intersecting families.

\section{Background}\label{Background}

\subsection{Representation theory}

Let $G$ be a finite group. A \emph{complex representation} of $G$ is a pair $(\rho, V)$, where $V$ is a finite-dimensional complex vector space and $\rho: G \rightarrow \GL(V)$ is a homomorphism. An \emph{equivalence} between two representations $(\rho, V)$ and $(\rho', V')$ is a linear isomorphism $f: V \rightarrow V'$ such that $f(\rho(g)(v)) = \rho'(g)(f(v))$ for all $g \in G$ and $v \in V$. The representation $(\rho, V)$ is said to be \emph{irreducible} if there is no non-zero proper subspace of $V$ which is $\rho(g)$-invariant for all $g \in G$. There are only finitely many equivalence classes of irreducible complex representations of $G$. The \emph{character} $\theta_\rho$ of $\rho$ is the map defined by $\theta_\rho(g) = \mathrm{Tr}(\rho(g))$. We denote by $\CR(G)$ the set of all complex irreducible characters of $G$. $\GL_2(q)$ has four types of irreducible complex representations, namely the \textbf{L}inear representations, the \textbf{S}teinberg representations, the \textbf{P}rincipal series representations, and the non-split \textbf{W}eil representations (also known as cuspidal representations). For more background, the readers may consult \cite[Section 4.1]{MR1431508}, which develops the finite-field model and its character formulas used in Theorem \ref{cusp} and \ref{table} below.

Fix an odd prime power $q$, a field $\BF$ with $q$ elements, $\BK/\BF$ be quadratic, and $\BK^1 := \cb{z \in \BK^*: z^{q + 1} = 1}$. Assume $\BF^* = \jk{\varepsilon}$ and $\BK^* = \jk{\omega}$ with $\omega^{q + 1} = \varepsilon$. Fix a nontrivial complex additive character $\psi$ of $\BF$ and a complex  multiplicative character $\chi$  of $\BK^*$ with $\abs{\chi(\BK^1)} \ne 1$. Let $\CS := \cb{\text{complex functions on } \BK}$. For $f \in \CS$, the \emph{Fourier transform} $\widehat{f}(x) := -q^{-1}\cdot\sum_{y \in \BK}f(y)\psi(x^qy + xy^q).$ Let $\CS_\chi := \cb{f \in \CS: f(yx) = f(x)/\chi(y)\; \forall y \in \BK^1}.$ The cuspidal representation is as follows.

\begin{theorem}\label{cusp}
$\GL_2(q) \curvearrowright \CS_\chi$ is irreducible, the action is described as follows:
\[\begin{aligned}
	&{\jz{a}{}{}{a^{-1}}f}\;\;(x) := f(ax);&\quad&{\jz{1}{b}{}{1}f}\;\;(x) := \psi(bx^{q + 1})f(x);\\
	&{\jz{}{1}{-1}{}f}\;\;(x) := \widehat{f}(x);&\quad&{\jz{\varepsilon}{}{}{1}f}\;\;(x) := \chi(\omega)f(\omega x).
\end{aligned}\]
	
\end{theorem}

Define $c_1(x) := \jz{x}{}{}{x}$, $c_2(x) := \jz{x}{1}{}{x}$, $c_3(x, y) := \jz{x}{}{}{y}$, $c_4(x + y\sqrt{\varepsilon}) := \jz{x}{y}{y\varepsilon}{x}$. A \emph{character table} of a finite group is a table whose rows are irreducible characters, and whose columns are conjugacy classes of this group. The following theorem gives the character table of $\GL_2(q)$.

\begin{theorem}\label{table}
Let $\alpha$ and $\beta$ run through all complex multiplicative characters of $\BF^*$ with $\alpha \ne \beta$, let $\lambda$ run through all complex multiplicative characters of $\BK^*$ with $\lambda \ne \lambda^q$. Let $x$ and $y$ run through all elements of $\BF^*$ with $x \ne y$, let $z$ run through all elements of $\BK^*$ with $z \ne z^q$. The character table of $\GL_2(q)$ is as follows: (The number in parentheses indicates the size of that class.
)
\[\begin{tabular}{|c|c|c|c|c|}
\toprule[0.1em]
\diagbox{Char}{Class}&$c_1(x)$($1$)&$c_2(x)$($q^2 - 1$)&$c_3(x, y)$($q^2 + q$)&$c_4(z)$($q^2 - q$)\\
\midrule[0.1em]
$L_\alpha$&$\alpha(x^2)$&$\alpha(x^2)$&$\alpha(xy)$&$\alpha(z^{q + 1})$\\
\hline
$S_\alpha$&$q\alpha(x^2)$&$0$&$\alpha(xy)$&$-\alpha(z^{q + 1})$\\
\hline
$P_{\alpha, \beta}$&$(q + 1)\alpha\beta(x)$&$\alpha\beta(x)$&$\alpha(x)\beta(y) + \alpha(y)\beta(x)$&$0$\\
\hline
$W_\lambda$&$(q - 1)\lambda(x)$&$-\lambda(x)$&$0$&$-\lambda(z) - \lambda(z^q)$\\
\bottomrule[0.1em]
\end{tabular}\]	
	
\end{theorem}

\subsection{Spectral graph theory}

For any square matrix $A$, we denote by $\spec(A)$ the set of all eigenvalues of $A$.

A real function $w$ on a finite group $G$ is \emph{symmetric} if $w(g^{-1}) = w(g)$ for any $g \in G$. A real function $w$ on $G$ can be regarded as a square matrix whose rows and columns are indexed by elements of $G$ and the $(x, y)$-entry of this matrix is $w(xy^{-1})$.

\begin{theorem}[Babai \cite{MR546860}]\label{Babai}
	
Let $w$ be a real symmetric class function on a finite group $G$. Then $\spec(w) = \cb{E_w(\theta) := \frac{1}{\theta(1)}\cdot\sum_{g \in G}\theta(g)w(g): \theta \in \CR(G)}$.
\end{theorem}

Let $\Gamma = (V, E)$ be a finite graph. A real function $w$ on $\binom{V}{2}$ is \emph{adjacent} if $\supp(w) \subseteq E$; is \emph{$d$-regular} if $\sum_{i}w(\cb{i, j})\equiv d > 0$ for any $j \in V$. A function $w$ on $\binom{V}{2}$ can be regarded as a symmetric zero-diagonal square matrix whose rows and columns are indexed by $V$ and the $(i, j)$-entry of this matrix is $w(\cb{i, j})$. Weights may have either sign; adjacency means only that weights vanishes on non-edges.

\begin{theorem}[Hoffman \cite{MR284373}]\label{Hoffman}
	
Let $\Gamma = (V, E)$ be a finite graph and $w$ be a real adjacent $(d - 1)$-regular function on $\binom{V}{2}$. If $\spec(w) \subseteq [-1, +\infty)$, then the independence number of $\Gamma$ is no larger than $\abs{V}/d$.
	
\end{theorem}

	

\section{Preliminaries}\label{Preliminaries}

Recall that $\CS_\chi$ is the cuspidal representation of $\GL_2(q)$. With respect to $\GL_2(q) \curvearrowright \CS_\chi$, define $\eta(g) := \dim_{\BC}\cb{v \in \CS_\chi: gv = v}$ where $g \in \GL_2(q)$. For $H \ls \BK^*$ and $z \in \BK^*$, we denote $\ord(zH) := \min\cb{d \gs 1: z^d \in H}$. In particular, $\ord(z) = \min\cb{d \gs 1: z^d = 1}$. If $P$ is a sentence, write $\one\fk{P} = 1$ if $P$ holds, otherwise $0$. In particular, if $a$ is a number or a map then write $\one\fk{a} := \one\fk{a = 1}$. Write $g_1\sim g_2$ if $g_1$ and $g_2$ are conjugate in $\GL_2(q)$. Throughout Sections \ref{Preliminaries} and \ref{Proof} we use the hypothesis of Theorem \ref{main} and put $\ell := \abs{\chi(\BF^*)}$, $m := (q - 1)/\ell = \abs{\ker\chi\cap\BF^*}$, $k := \min\cb{k' \gs t: k' \mid q - 1}$. Discrete logarithms are regarded as elements of $\BZ/(q - 1)\BZ$ or $\BZ/(q^2 - 1)\BZ$ respectively. In divisibility statements we use any integer lift; all moduli used divide the corresponding group order.

\subsection{Calculation of $\eta$}
We begin by computing the values of $\eta$ on the conjugacy classes, since $\eta$ determines whether two group elements are $t$-intersecting. Rather than working directly from the definition of $\eta$, we use the following convenient formula $\eta(g) = \frac{1}{\ord(g)}\cdot\sum\limits_{i = 1}^{\ord(g)}W_\chi(g^i)$.

\begin{lemma}\label{eta}
	
With respect to $\GL_2(q) \curvearrowright \CS_\chi$, for any $g \in \GL_2(q)$, we have:
	\[\eta(g) = \left\{\begin{aligned}
		&(q - 1)\cdot\one\fk{\chi(x)},&&\text{ if } g \sim c_1(x);\\
		&\pt{q/p - 1}\cdot\one\fk{\chi(x)},&&\text{ if } g \sim c_2(x);\\
		&(q - 1)\cdot\ord(x/y)^{-1}\cdot\one\fk{\chi(y)^{\ord(x/y)}},&&\text{ if } g \sim c_3(x, y);\\
		&(q + 1)\cdot\ord(z\BF^*)^{-1}\cdot\one\fk{\chi(z)^{\ord(z\BF^*)}} - 2\cdot\one\fk{\chi(z)},&&\text{ if } g \sim c_4(z).
	\end{aligned}\right.\]	
	
\end{lemma}

\begin{proof}
	
From the fourth row of the table in Theorem \ref{table}, we have the following cases:
 
\noindent\textbf{Case 1:} $g \sim c_1(x)$ for some $x \in \BF^*$. Note that $\ord(c_1(x)) = \ord(x)$, thus
\[\eta(c_1(x)) = \frac{1}{\ord(x)}\cdot\sum_{i = 1}^{\ord(x)}(q - 1)\chi(x)^i = (q - 1)\cdot\one\fk{\chi(x)}.\]
\textbf{Case 2:} $g \sim c_2(x)$ for some $x \in \BF^*$. Note that $\ord(c_2(x)) = p\cdot\ord(x)$ and $c_2(x)^i \sim c_1(x^i)$ if $p \mid i$ and $c_2(x)^i \sim c_2(x^i)$ if $p \nmid i$. Also note that $\one\fk{\chi(x)} = \one\fk{\chi^p(x)}$ since $\gcd(p, q \pm 1) = 1$, thus
\[\eta(c_2(x)) = \frac{1}{p\cdot\ord(x)}\cdot\cb{\sum_{i = 1}^{\ord(x)}(q - 1)\chi(x)^{pi} - \sum^{p\cdot\ord(x)}_{i = 1}\chi(x)^i + \sum_{i = 1}^{\ord(x)}\chi(x)^{pi}} = \pt{q/p - 1}\cdot\one\fk{\chi(x)}.\]
\textbf{Case 3:} $g \sim c_3(x, y)$ for some $x, y \in \BF^*$ with $x \ne y$. Put $r := x/y$ and $s := y^{\ord(x/y)}$, so $\ord(c_3(x, y)) = \ord(r)\cdot\ord(s)$, thus
\[\eta(c_3(x, y)) = \frac{1}{\ord(r)\cdot\ord(s)}\cdot\sum_{i = 1}^{\ord(s)}(q - 1)\chi(s)^i = \frac{q - 1}{\ord(r)}\cdot\one\fk{\chi(s)}.\]
\textbf{Case 4:} $g \sim c_4(z)$ for some $z \in \BK^*$ with $z \ne z^q$. Put $e := \ord(z\BF^*)$. Note that $\one\fk{\chi(z)} = \one\fk{\chi^p(z)}$ since $\gcd(p, q\pm1) = 1$, thus
\[\begin{aligned}
	\eta(c_4(z))
	&= \frac{1}{\ord(z)}\cdot\cb{\sum_{i = 1}^{\ord(z)/e}(q - 1)\chi(z)^{ei} - \sum_{i = 1}^{\ord(z)}\pt{\chi(z)^i + \chi^q(z)^i} + \sum_{i = 1}^{\ord(z)/e}\pt{\chi(z)^{ei} + \chi^q(z)^{ei}}}\\
	&= \frac{q + 1}{e}\cdot\one\fk{\chi(z)^e} - 2\cdot\one\fk{\chi(z)}.
\end{aligned}\]
The proof is completed.
\end{proof}

\subsection{Constructions of $t$-intersecting families}
Having computed $\eta$, we now construct several $t$-intersecting families to show that the extremal value in the Theorem \ref{main} is attained. In fact, these families are subgroups, so it suffices to verify that $\eta(g) \gs t$ for every element $g$. Recall $\ell := \abs{\chi(\BF^*)}$ and $\abs{\chi(\BK^1)} \ne 1$.

\begin{lemma}\label{K1}
	
$K_1 := \cb{g \in \GL_2(q): \det(g) \in \ker\chi\cap\BF^*}$ is $1$-intersecting if $2 \nmid \ell$. Moreover, $\abs{K_1} = \abs{\GL_2(q)}/\ell$.
	
\end{lemma}

\begin{proof}

Above all, we claim: $q + 1 \gs 3\cdot\ord(z\BF^*)$ if $\chi(z) = 1$. Reason: Put $n := \abs{\chi(\BK^*)}$. Since $2 \nmid \ell$, $2\ell\mid q - 1$; therefore, if $n/\ell \in \cb{1, 2}$, then $n \mid q - 1$, making $\abs{\chi(\BK^1)} = 1$, a contradiction, hence $n/\ell \gs 3$. Now the map $x \mapsto x^{q - 1}$ on $\ker\chi$ has kernel $\ker\chi \cap \BF^*$, so its image has order $\frac{(q^2 - 1)/n}{(q - 1)/\ell} = \frac{\ell(q + 1)}{n}$. Then $\ord(z\BF^*) = \ord(z^{q - 1}) \ls \ell(q + 1)/n \ls (q + 1)/3$.
	
Firstly suppose $g \sim c_i(x) \in K_1$ where $i \in \cb{1, 2}$, thus $\chi(x)^2 = \chi(x^2) = \chi(\det g) = 1$, but $2 \nmid \ell$ implies $\chi(x) = 1$, so $\eta(c_i(x)) \gs 1$ by Lemma \ref{eta}. Secondly suppose $g \sim c_3(x, y) \in K_1$, put $r := x/y$ and $s := x^{\ord(r)} = y^{\ord(r)}$. Since $\chi(x)\chi(y) = \chi(xy) = \chi(\det g) = 1$, we have $\chi(x) = \chi(y)^{-1}$, substituting this in $\chi(x)^{\ord(r)} = \chi(s) = \chi(y)^{\ord(r)}$ gives $\chi(s)^2 = 1$, but $2 \nmid \ell$ implies $\chi(s) = 1$, so $\eta(c_3(x, y)) \gs 1$ by Lemma \ref{eta}. Finally suppose $g \sim c_4(z) \in K_1$, put $e := \ord(z\BF^*)$.  Since $\chi(z)^{2e} = \chi((z^e)^{q + 1}) = \chi((z^{q + 1})^e) = \chi(\det g)^e = 1$, but $2 \nmid \ell$ implies $\chi(z)^e = 1$, so Lemma \ref{eta} and the claim above give $\eta(c_4(z)) = \frac{q + 1}{e} - 2\cdot\one[\chi(z)] \gs 1$.
\end{proof}

\begin{lemma}\label{K2}
	
$K_2 := \cb{g \in \GL_2(q): \det(g) \in (\ker\chi\cap\BF^*)^2}$	is $2$-intersecting if $\chi(-1) = 1$. Moreover, $\abs{K_2} = \abs{\GL_2(q)}/2\ell$.
	
\end{lemma}

\begin{proof}

Above all, we claim: $q + 1 \gs 4 \cdot \ord(z\BF^*)$ if $z \in \BK^1$ and $\chi(z) = 1$. Reason: Put $n := \abs{\chi(\BK^1)} \gs 2$. $\ker\chi\cap\BK^1$ is cyclic of order $(q + 1)/n$ and contains $-1$, so its order is even and its subgroup of squares, has order $(q + 1)/2n$ and contains $z^{q - 1} = z^{-2}$. Consequently, $\ord(z\BF^*) = \ord(z^{q - 1}) \ls (q + 1)/2n \ls (q + 1)/4$.
	
Let $g \in K_2$, thus $\det g = x^2$ for some $x \in \ker\chi \cap \BF^*$, note that $c_1(x)$ acts trivially on $\CS_\chi$, so $\eta(g) = \eta(x^{-1}g)$. It suffices to prove $\eta(g) \gs 2$ for any $g \in K_2 \cap \SL_2(q)$. 

Firstly suppose $g \sim c_i(\pm 1) \in K_2$ where $i \in \cb{1, 2}$, so $\eta(c_i(\pm1)) \gs 2$ by Lemma \ref{eta} and $\chi(-1) = 1$. Secondly suppose $g \sim c_3(x, x^{-1}) \in K_2$, note that $x^{\ord(x^2)} = \pm 1$, so $\eta(c_3(x, x^{-1})) = \frac{q - 1}{{\ord(x^2)}} \gs 2$ by Lemma \ref{eta} and $\chi(-1) = 1$. Finally suppose $g \sim c_4(z) \in K_2$ where $z \in \BK^1$, put $e := \ord(z\BF^*)$. Since $z^e \in \BK^1 \cap \BF^* = \cb{\pm1}$, so $\eta(c_4(z)) = \frac{q + 1}{e} - 2\cdot\one\fk{\chi(z)} \gs 2$ by Lemma \ref{eta}, the claim above and $\chi(-1) = 1$.
\end{proof}

Both sets are determinant preimages of subgroups. Since the determinant map is onto and $m = \abs{\ker\chi\cap\BF^*}$, the first, $K_1$, has index $\ell$. For the second, $K_2$, $\chi(-1) = 1$ implies $-1 \in \ker\chi\cap\BF^*$, so $m$ is even and its square subgroup has order $m/2$, giving index $2\ell$.

Recall $k = k(t) := \min\cb{k' \gs t: k'\mid q - 1}$ and put $(\BF^*)^k := \cb{x^k: x \in \BF^*}$.
\begin{lemma}\label{Borel}
	
$B_t := \cb{\jz{a}{b}{}{d} \in \GL_2(q): a/d\in(\BF^*)^k \text{ and } \chi(d) = 1}$ is $t$-intersecting. Moreover,  $\abs{B_t} = \abs{\GL_2(q)}/k\ell(q + 1)$.
	
\end{lemma}

\begin{proof}
	
The diagonal maps $g \mapsto d$ and $g \mapsto a/d$ are homomorphisms on the upper triangular subgroup; hence $B_t$ is a subgroup. Its order is $\abs{B_t} = qm\cdot\frac{q - 1}{k} = \abs{\GL_2(q)}/k\ell(q + 1)$. Let $g \in B_t$ with diagonal elements $x$ and $y$. Firstly suppose $x = y$, thus $g \sim c_i(x)$ where $i \in \cb{1, 2}$, so $\eta(c_i(x)) \gs t$ by Lemma \ref{eta}, $\chi(x) = 1$ and $p > t$. Secondly suppose $x \ne y$, thus $g \sim c_3(x, y)$, so $\eta(c_3(x, y)) = \frac{q - 1}{\ord(x/y)} \gs \frac{q - 1}{\abs{(\BF^*)^k}} = k \gs t$ by Lemma \ref{eta} and $\chi(y) = 1$.
\end{proof}

\subsection{Calculations of character sums}

As Theorem \ref{Babai} shows, computing the eigenvalues involves evaluating numerous character sums. In this subsection, we evaluate several such sums. Recall that $\BF^* = \jk{\varepsilon}$, $\BK^* = \jk{\omega}$ and $\ell := \abs{\chi(\BF^*)}$.

\begin{lemma}\label{division}
	
(1) For $x \in \BF^*$, $\chi(x) = 1$ if and only if $\ell \mid \log_\varepsilon x$. (2) For $z \in \BK^*$, $2 \mid \ell$ and $\chi(z)^{\ord(z\BF^*)} = 1$ implies $2\ell \mid \log_\omega z$.
	
\end{lemma}

\begin{proof}
	
(1) The assertion follows from $\chi(x) = \chi(\varepsilon)^{\log_\varepsilon x}$ and $\ell = \abs{\chi(\BF^*)}$. For (2), choose an integer lift $a = \log_\omega z$ and put $d = \gcd(a, q + 1)$. Then $e := \ord(z\BF^*) = (q + 1)/d$ and $z^e = \varepsilon^{a/d}$. By (1), $\chi(z)^e = 1$ is equivalent to $\ell \mid \frac{a}{d}$. If $d$ were odd, $(q + 1)/d$ would be even, and $\gcd(a/d, (q + 1)/d) = 1$ would force $a/d$ to be odd. This contradicts even $\ell$. Thus $d$ is even, and $a = d(a/d)$ is divisible by $2\ell$. 
\end{proof}

\begin{remark}\label{rem1}
	
For a complex multiplicative character $\alpha$ of $\BF^*$, we have $\one\fk{\alpha^{2\ell}} = \one\fk{\alpha^\ell}$ if $2 \nmid m$ since $(\alpha^\ell)^m = 1$, thus
\[\sum_{z \in \BK^*: \log_\omega z \equiv \ell\pmod{2\ell}}\alpha(z^{q + 1})
	= \sum_{i = 1}^{(q^2 - 1)/2\ell}\alpha(\varepsilon)^{\ell(2i + 1)}\\
	= \frac{m(q + 1)}{2}\cdot\one\fk{\alpha^\ell}.\]
and
\[\sum_{z \in \BK^*: \log_\omega z \not\equiv 0 \pmod{\ell}}\alpha(z^{q + 1})
	= \sum_{i = 1}^{q^2 - 1}\alpha(\varepsilon)^i - \sum_{i = 1}^{(q^2 - 1)/\ell}\alpha(\varepsilon)^{\ell i}\\
	= m(q + 1)\pt{\ell\cdot\one\fk{\alpha} - \one\fk{\alpha^\ell}}.\]	
	
\end{remark}

\begin{remark}\label{rem2}
	
For a complex multiplicative character $\lambda$ of $\BK^*$ with $\lambda \ne \lambda^q$, we have $\lambda^\ell \ne 1$ since $\ell \mid q - 1$, thus
\[\sum_{z \in \BK^*: \log_\omega z\equiv \ell\pmod{2\ell}}\lambda(z)
= \sum_{i = 1}^{(q^2 - 1)/2\ell}\lambda(\omega)^{\ell(2i + 1)}
= -\frac{m(q + 1)}{2}\cdot\one\fk{\lambda^{2\ell}}.\]
and
\[\sum_{z \in \BK^*: \log_\omega z\not\equiv0\pmod{\ell}}\lambda(z) 
= \sum_{i = 1}^{q^2 - 1}\lambda(\omega)^i - \sum_{i = 1}^{(q^2 - 1)/\ell}\lambda(\omega)^{\ell i}
= 0.\]	
\end{remark}

\begin{remark}\label{rem3}
	
Let $\sigma_\alpha := \frac{1}{k}\cdot\abs{\sum\limits_{i = 0}^{k - 1}\alpha(\varepsilon)^i}^2$, thus $\sigma_1 = k$. Also $0 \ls \sigma_\alpha \ls k$. Define $$U_x := \cb{(i, j): 0 \ls i, j < k \text{ and } i \ne j \text{ and } \varepsilon^{i - j} = x}.$$Swapping $i, j$ gives $\abs{U_x} = \abs{U_{x^{-1}}}$. In particular the definition does not discard negative differences. For two complex multiplicative characters $\alpha$ and $\beta$ of $\BF^*$, we have
	\[\begin{aligned}
		\sum_{x, y \in \BF^*: x \ne y}\abs{U_{x/y}}\alpha(x)\beta(y)
		&= \sum_{x, y \in \BF^*: x \ne y}\abs{U_{x/y}}\alpha(x/y)\cdot\alpha\beta(y)\\
		&= \sum_{r \in \BF^*: r \ne 1}\abs{U_r}\alpha(r)\cdot\sum_{y \in \BF^*}\alpha\beta(y)\\
		&= (q - 1)\cdot\one\fk{\alpha\beta}\cdot\sum_{r \in \BF^*: r \ne 1}\abs{U_r}\alpha(r)\\
		&= (q - 1)\cdot\one\fk{\alpha\beta}\cdot\fk{\pt{\sum_{i = 0}^{k - 1}\alpha(\varepsilon)^i}\pt{\sum_{j = 0}^{k - 1}\alpha(\varepsilon)^{-j}} - k}\\
		&= k(q - 1)(\sigma_\alpha - 1)\cdot\one\fk{\alpha\beta}.
	\end{aligned}\]Here $i \ne j$ implies $\varepsilon^{i - j} \ne 1$, since $0 < \abs{i - j} < k \ls q - 1$.
\end{remark}

\section{Complete Proof of Theorem \ref{main}}\label{Proof}

Let $\Gamma := \Cay(\GL_2(q), \cb{g \in \GL_2(q): \eta(g) < t})$. Then $S \subseteq \GL_2(q)$ is $t$-intersecting if and only if $S$ is a coclique in $\Gamma$. By Theorem \ref{Hoffman}, if we want to prove $\abs{S} \ls \abs{\GL_2(q)}/d$, it suffices to show there is a real adjacent $(d - 1)$-regular function on $\Gamma$ such that $\spec(w) \subseteq [-1, +\infty)$. The function $\eta$ is invariant under conjugacy and satisfies $\eta(g^{-1}) = \eta(g)$. Also $\eta(1) = q - 1 \gs t$. Thus $\Gamma$ is a simple undirected normal Cayley graph, and left translations preserves its edges. In each case below we identify the attaining subgroup before proving the matching upper bound.

\subsection{The cases $2\nmid\ell$ and $t = 1$}\label{case1}

Lemma \ref{K1} gives the lower bound $\abs{\GL_2(q)}/\ell$. If $\ell = 1$, the whole group $\GL_2(q)$ is already $1$-intersecting, and the result follows. Assume $\ell > 1$. Let $$w(g) := \frac{1}{m}\cdot\one\fk{g = c_1(x): \log_\varepsilon x \not\equiv 0\pmod{\ell}}.$$Note that $w$ is adjacent and $(\ell - 1)$-regular for $\Gamma$ by Lemma \ref{eta} and \ref{division}. Also note that it is a real symmetric class function, so we can apply Theorem \ref{Babai}. For any $\theta \in \CR(\GL_2(q))$, there exists a complex multiplicative character $\alpha_\theta$ of $\BF^*$, such that $\theta(c_1(x)) = \theta(1)\alpha_\theta(x)$. Hence$$mE_w(\theta) = \sum_{x \in \BF^*: \ell\nmid\log_\varepsilon x}\alpha_\theta(x) = \sum_{x \in \BF^*}\alpha_\theta(x) - \sum_{x \in \BF^*: \ell\mid\log_\varepsilon x}\alpha_\theta(x) = m\pt{\ell\cdot\one\fk{\alpha_\theta} - \one\fk{\alpha_\theta^\ell}}.$$So $\spec(w) = \cb{\ell - 1, 0, -1} \subseteq [-1, +\infty)$. Theorem \ref{Hoffman} gives $\abs{S} \ls \abs{\GL_2(q)}/\ell$, matching Lemma \ref{K1}. $\hfill\blacksquare$

\subsection{The cases $2 \mid \ell$, $\chi(-1) = 1$ and $t = 1$}\label{case2}

Lemma \ref{K2} provides a $2$-intersecting, hence $1$-intersecting subgroup of order $\abs{\GL_2(q)}/2\ell$. Let $$w(g) := \frac{2}{m(q + 1)}\cdot\one\fk{g= c_4(z): \log_\omega z\not\equiv0\pmod{2\ell}}.$$Note that $w$ is adjacent and $(2\ell - 1)$-regular for $\Gamma$ by Lemma \ref{eta} and \ref{division}. Set $G := \GL_2(q)$, $T := \cb{c_4(z): z \in \BK^*}$ and $H := \cb{c_4(z): z \in \BK^* \text{ and } 2\ell\mid\log_\omega z}$. This weight is supported on $T \setminus H$ in one fixed non-split torus; it need not be a class function. We calculate its spectrum directly rather than applying Theorem \ref{Babai}. In the support subgraph, each right coset of $T$ is a complete $2\ell$-partite graph, whose parts are cosets of $H$ of size $\abs{H} = m(q + 1)/2$. If $I_n$ and $J_n$ denote the identity and all-one matrices of order $n$, respectively, then, after ordering vertices by these cosets, $$w = \abs{H}^{-1}\cdot I_{\abs{G}/\abs{T}}\otimes(J_{2\ell} - I_{2\ell})\otimes J_{\abs{H}}.$$Hence $\spec(w) = \cb{2\ell - 1, 0, -1} \subseteq [-1, +\infty)$. Theorem \ref{Hoffman} gives the upper bound $\abs{\GL_2(q)}/2\ell$, equal to the construction in Lemma \ref{K2}.   $\hfill\blacksquare$

\subsection{The cases $\chi(-1) = 1$ and $t = 2$}\label{case3}

Lemma \ref{Borel} gives the lower bound $\abs{\GL_2(q)}/2\ell$. We prove the upper bound by an explicit clique and double counting. Let $C := \cb{c_1(\varepsilon^i), c_3(\varepsilon^{i + 1}, \varepsilon^i): 0 \ls i < \ell}$, thus $\abs{C} = 2\ell$. A quotient of two distinct $c_1$-elements, or two distinct $c_3$-elements, is $c_1(\varepsilon^{i - j})$, where $0 < \abs{i - j} < \ell$. Therefore $\chi(\varepsilon^{i - j}) \ne 1$, and Lemma \ref{eta} gives $\eta = 0$. A mixed quotient is a scalar multiple of $c_3(\varepsilon, 1)$ or $c_3(1/\varepsilon, 1)$, whose eigenvalue ratio has order $q - 1$. Since every character value on $\BF^*$ has order dividing $q - 1$, Lemma \ref{eta} gives $\eta = 1$. Thus every pair of distinct elements of $C$ is an edge of $\Gamma$, and $C$ is a clique of size $2\ell$.

For any $2$-intersecting family $S \subseteq \GL_2(q)$, without loss of generality, we may assume $1 \in S$. Then $\eta(g) \gs 2$ for any $g \in S$, and $gC \cap C = \varnothing$ for any $g \in S \setminus \cb{1}$. This yields$$\abs{S}\cdot\abs{C} = \sum_{g \in S}\abs{gC} \ls \abs{\GL_2(q)}$$hence $\abs{S} \ls \abs{\GL_2(q)}/2\ell$. This matches the lower bound. $\hfill\blacksquare$

\subsection{The cases $\chi(-1) = -1$}\label{case4}
	
Lemma \ref{Borel} gives the lower bound $\abs{\GL_2(q)}/k\ell(q + 1)$. Here $\ell$ is even and $m$ is odd. The weight below is designed so that the split term uses the length-$k$ cyclic correlation $\abs{U_{x/y}}$ to control element orders. The non-split term cancels unfavorable character sums, while the scalar and non-semisimple terms balance the remaining character contributions. We verify well-definedness and support first, then the row sum and spectral lower bound. For $z \in \BK^*$, define
\[u(z) := \left\{\begin{aligned}
	&0,&\quad&\text{if } \log_\omega z \equiv 0 \pmod{2\ell};\\
	&\frac{2}{mq}\cdot\pt{\frac{k}{q - 1} + \frac{1}{q + 1}},&\quad&\text{if } \log_\omega z \equiv \ell \pmod{2\ell};\\
	&\frac{k}{mq(q - 1)},&\quad&\text{if } \log_\omega z \not\equiv 0 \pmod{\ell}.
\end{aligned}\right.\]

Recall Remark \ref{rem3}. Define
\[w(g) := \left\{\begin{aligned}
	&\frac{q + 1}{2mq}\cdot\one\fk{\chi(x) \ne 1} - \frac{q - 1}{2}\cdot u(x),&\quad&\text{if } g\sim c_1(x);\\
	&\frac{1}{2mq}\cdot\one\fk{\chi(x) \ne 1} + \frac{1}{2}\cdot u(x),&\quad&\text{if } g\sim c_2(x);\\
	&\frac{1}{kmq}\cdot\abs{U_{x/y}},&\quad&\text{if } g\sim c_3(x, y);\\
	&u(z),&\quad&\text{if } g\sim c_4(z).
\end{aligned}\right.\]
This is a well-defined symmetric real class function. For the split classes, $\abs{U_r} = \abs{U_{r^{-1}}}$ handles both interchange of $x, y$ and inversion. For the non-split classes, multiplication of $\log_\omega z$ by $q$ or by $-1$ preserves the three cases defining $u$: use $q \equiv 1 \pmod{\ell}$, odd $q$, and even $\ell$. Thus $u(z^q) = u(z) = u(z^{-1})$. The scalar and non-semisimple cases are also invariant under inversion. Some weights may be negative, as permitted by Theorem \ref{Hoffman}.
\begin{lemma}
	
$w$ is adjacent for $\Gamma$, that is, $w(g) = 0$ if $\eta(g) \gs t$.	
	
\end{lemma}

\begin{proof}
	
For $i \in \cb{1, 2}$, the inequalities $q/p - 1 \gs t$ and Lemma \ref{eta} give $\eta(c_i(x)) \gs t$ if and only if $\chi(x) = 1$. Then $\ell \mid \log_\varepsilon x$ by Lemma \ref{division}, and $2\ell \mid (q + 1)\log_\varepsilon x = \log_\omega x$, so $u(x) = 0$ and $w(c_i(x)) = 0$. 

For the split case, suppose $\abs{U_{r}} > 0$, where $r = x/y$. Then $r = \varepsilon^d$ for some $d = i - j$ with $0 < \abs{d} \ls k - 1$. Put $d' = \gcd(d, q - 1)$. We have $d' \mid q - 1$ and $d' < k$, so the definition of $k$ implies $d' < t$. Also$$\ord(r) = \frac{q - 1}{d'},\quad \eta(c_3(x, y)) \in \cb{0, d'}$$ by Lemma \ref{eta}. Hence a non-zero split weight implies $\eta(c_3(x, y)) < t$. This proves the required vanishing.

For $g \sim c_4(z)$, put $e = \ord(z\BF^*)$. If $\eta(g) \gs t$, Lemma \ref{eta} forces $\chi(z)^e = 1$. Since $\ell$ is even, Lemma \ref{division} gives $2\ell \mid \log_\omega z$, and thus $u(z) = 0$.
\end{proof}

\begin{proposition}\label{last}
	
Let $\alpha$ and $\beta$ be two complex multiplicative characters of $\BF^*$ with $\alpha \ne \beta$, let $\lambda$ be a complex multiplicative character of $\BK^*$ with $\lambda \ne \lambda^q$. The eigenvalues of $w$ are as follows:
\[E_w(\theta) = \left\{\begin{aligned}
	&\frac{\ell(q + 1)}{2}\cdot\pt{\sigma_\alpha\cdot\one\fk{\alpha^2} + k\cdot\one\fk{\alpha}} - \one\fk{\alpha^\ell},&\quad&\text{if } \theta = L_\alpha;\\
	&\frac{\ell(q + 1)}{2q}\cdot\pt{\sigma_\alpha\cdot\one\fk{\alpha^2} - k\cdot\one\fk{\alpha}} - \one\fk{\alpha^\ell},&\quad&\text{if } \theta = S_\alpha;\\
	&\ell\sigma_\alpha\cdot\one\fk{\alpha\beta} - \one\fk{(\alpha\beta)^\ell},&\quad&\text{if } \theta = P_{\alpha, \beta};\\
	&\pt{1 + \frac{k(q + 1)}{q - 1}}\cdot\one\fk{\lambda^{2\ell}},&\quad&\text{if } \theta = W_\lambda.
\end{aligned}\right.\]
	
\end{proposition}

\begin{remark}

By Remark \ref{rem3}, $0 \ls \sigma_\alpha \ls k$ and $\sigma_1 = k$. In the $L_\alpha$ and $P_{\alpha, \beta}$ rows, every term except the final negative indicator is non-negative, so the value is at least $-1$. The $W_\lambda$ row is non-negative. In the Steinberg row, if $\alpha = 1$ then $\sigma_1 - k = 0$, whereas if $\alpha \ne 1$ the negative term $-k\cdot\one\fk{\alpha}$ vanishes. That row is also at least $-1$. In particular, $$E_w(L_1) = k\ell(q + 1) - 1,\quad E_w(S_1) = -1.$$Thus $w$ has row sum $k\ell(q + 1) - 1$ and least eigenvalue $-1$. Once Proposition \ref{last} is proved, Theorem \ref{Hoffman} yields $\abs{S} \ls \abs{\GL_2(q)}/k\ell(q + 1)$, which matches Lemma \ref{Borel} and completes Theorem \ref{main}.
	
\end{remark}

\begin{proof}

By the definition of $u$ and Remarks \ref{rem1}--\ref{rem3}, we have$$\sum_{z \in \BK^*}\alpha(z^{q + 1})u(z) = \frac{k(q + 1)}{mq}\cdot\one\fk{\alpha} + \frac{1}{q}\cdot\one\fk{\alpha^\ell}$$and$$\sum_{z \in \BK^*}\lambda(z)u(z) = \frac{-1}{q}\cdot\pt{1 + \frac{k(q + 1)}{q - 1}}\cdot\one\fk{\lambda^{2\ell}}.$$	
Note that the size for $c_3(x, y)$ (resp. $c_4(z)$) is $q^2 + q$ (resp. $q^2 - q$). When summing over all ordered pairs $x \ne y$, use $\frac{q^2 + q}{2}$; when summing over all $z \in \BK^* \setminus \BF^*$, use $\frac{q^2 - q}{2}$, because these parametrizations count each class twice. In every row, Theorem \ref{Babai} also divides the character sum by its degree $\theta(1)$, namely $1$, $q$, $q + 1$ and $q - 1$ respectively. We retain the intermediate sums to make all class-size and degree normalizations explicit.

For the linear characters, Theorem \ref{Babai} gives 
\[\begin{aligned}
E_w(L_\alpha)
&= \sum_{x \in \BF^*}\alpha(x^2)\pt{\frac{q + 1}{2mq}\cdot\one\fk{\chi(x) \ne 1} - \frac{q - 1}{2}\cdot u(x)}\\
&\;\;\;\;+ (q^2 - 1)\cdot\sum_{x \in \BF^*}\alpha(x^2)\pt{\frac{1}{2mq}\cdot\one\fk{\chi(x) \ne 1} + \frac{1}{2}\cdot u(x)}\\
&\;\;\;\;+ \frac{q^2 + q}{2}\cdot\sum_{x, y \in \BF^*: x \ne y}\alpha(xy)\cdot\frac{1}{kmq}\cdot\abs{U_{x/y}}\\
&\;\;\;\;+ \frac{q^2 - q}{2}\cdot\sum_{z \in \BK^* \setminus \BF^*}\alpha(z^{q + 1})u(z)\\
&= \frac{q + 1}{2m}\cdot\sum_{x \in \BF^*: \chi(x) \ne 1}\alpha(x^2) + \frac{q + 1}{2km}\cdot\sum_{x, y \in \BF^*: x \ne y}\alpha(xy)\abs{U_{x/y}}\\
&\;\;\;\;+ \frac{q^2 - q}{2}\cdot\sum_{z \in \BK^*}\alpha(z^{q + 1})u(z)\\
&= \frac{q + 1}{2m}\cdot\pt{(q - 1)\cdot\one\fk{\alpha^2} - m\cdot\one\fk{\alpha^\ell}} + \frac{q + 1}{2km}\cdot k(q - 1)(\sigma_\alpha - 1)\cdot\one\fk{\alpha^2}\\
&\;\;\;\;+\frac{q(q - 1)}{2}\cdot\pt{\frac{k(q + 1)}{mq}\cdot\one\fk{\alpha} + \frac{1}{q}\cdot\one\fk{\alpha^\ell}}\\
&= \frac{\ell(q + 1)}{2}\cdot\pt{\sigma_\alpha\cdot\one\fk{\alpha^2} + k\cdot\one\fk{\alpha}} - \one\fk{\alpha^\ell}.
\end{aligned}\]
For the Steinberg characters,
\[\begin{aligned}
E_w(S_\alpha)
&= \frac{1}{q}\cdot\left\{\sum_{x \in \BF^*}q\alpha(x^2)\pt{\frac{q + 1}{2mq}\cdot\one\fk{\chi(x) \ne 1} - \frac{q - 1}{2}\cdot u(x)}\right.\\
&\;\;\;\;+ \frac{q^2 + q}{2}\cdot\sum_{x, y \in \BF^*: x \ne y}\alpha(xy)\cdot\frac{1}{kmq}\cdot\abs{U_{x/y}}\\
&\;\;\;\;\left.- \frac{q^2 - q}{2}\cdot\sum_{z \in \BK^* \setminus \BF^*}\alpha(z^{q + 1})u(z)\right\}\\
&= \frac{q + 1}{2mq}\cdot\sum_{x \in \BF^*: \chi(x) \ne 1}\alpha(x^2) + \frac{q + 1}{2kmq}\cdot\sum_{x, y \in \BF^*: x \ne y}\alpha(xy)\abs{U_{x/y}} - \frac{q - 1}{2}\cdot\sum_{z \in \BK^*}\alpha(z^{q + 1})u(z)\\
&= \frac{q + 1}{2mq}\cdot\left((q - 1)\cdot\one\fk{\alpha^2} - m\cdot\one\fk{\alpha^\ell}\right) + \frac{q + 1}{2kmq}\cdot k(q - 1)(\sigma_\alpha - 1)\cdot\one\fk{\alpha^2}\\
&\;\;\;\;- \frac{q - 1}{2}\cdot\pt{\frac{k(q + 1)}{mq}\cdot\one\fk{\alpha} + \frac{1}{q}\cdot\one\fk{\alpha^\ell}}\\
&= \frac{\ell(q + 1)}{2q}\cdot\pt{\sigma_\alpha\cdot\one\fk{\alpha^2} - k\cdot\one\fk{\alpha}} - \one\fk{\alpha^\ell}.
\end{aligned}\]
For the principal-series characters,
\[\begin{aligned}
E_w(P_{\alpha, \beta})
&= \frac{1}{q + 1}\cdot\left\{\sum_{x \in \BF^*}(q + 1)\alpha\beta(x)\pt{\frac{q + 1}{2mq}\cdot\one\fk{\chi(x) \ne 1} - \frac{q - 1}{2}\cdot u(x)}\right.\\
&\;\;\;\;+ (q^2 - 1)\cdot\sum_{x \in \BF^*}\alpha\beta(x)\pt{\frac{1}{2mq}\cdot\one\fk{\chi(x) \ne 1} + \frac{1}{2}\cdot u(x)}\\
&\;\;\;\;\left.+ \frac{q^2 + q}{2}\cdot\sum_{x, y \in \BF^*: x \ne y}(\alpha(x)\beta(y) + \alpha(y)\beta(x))\cdot\frac{1}{kmq}\cdot\abs{U_{x/y}}\right\}\\
&= \frac{1}{m}\cdot\sum_{x \in \BF^*: \chi(x) \ne 1}\alpha\beta(x) + \frac{1}{km}\cdot\sum_{x, y \in \BF^*: x \ne y}\alpha(x)\beta(y)\abs{U_{x/y}}\\
&= \frac{1}{m}\cdot\pt{(q - 1)\cdot\one\fk{\alpha\beta} - m\cdot\one\fk{(\alpha\beta)^\ell}} + \frac{1}{km}\cdot k(q - 1)(\sigma_\alpha - 1)\cdot\one\fk{\alpha\beta}\\
&= \ell\sigma_\alpha\cdot\one\fk{\alpha\beta} - \one\fk{(\alpha\beta)^\ell}.
\end{aligned}\]
Finally, for the cuspidal characters,
\[\begin{aligned}
E_w(W_\lambda)
&= \frac{1}{q - 1}\cdot\left\{\sum_{x \in \BF^*}(q - 1)\lambda(x)\pt{\frac{q + 1}{2mq}\cdot\one\fk{\chi(x) \ne 1} - \frac{q - 1}{2}\cdot u(x)}\right.\\
&\;\;\;\;- (q^2 - 1)\cdot\sum_{x \in \BF^*}\lambda(x)\pt{\frac{1}{2mq}\cdot\one\fk{\chi(x) \ne 1} + \frac{1}{2}\cdot u(x)}\\
&\;\;\;\;\left.- \frac{q^2 - q}{2}\cdot\sum_{z \in \BK^* \setminus \BF^*}(\lambda(z) + \lambda^q(z))u(z)\right\}\\
&= -q\cdot\sum_{z \in \BK^*}\lambda(z)u(z)\\
&= \pt{1 + \frac{k(q + 1)}{q - 1}}\cdot\one\fk{\lambda^{2\ell}}.
\end{aligned}\]
This proves the four formulas.
\end{proof}

\section{Further Directions}

The theorem determines extremal sizes in its state regimes, rather than a classification of all maximum families. Determining the value for $\chi(-1) = 1$, $t \gs 3$ extending the analysis to $q = p$ or $p \ls t$, and studying uniqueness, the least-eigenspace structure or stability of maximum families remain separate questions.

\vspace{0.5cm}
\newpage
\noindent\textbf{\Large Acknowledgement} \vspace{0.5cm}

The author acknowledges the use of ChatGPT 6.1 Sol during the exploratory stage of this project. All mathematical arguments and proofs presented in the final manuscript were developed and rigorously verified by the author. The author takes full responsibility for the content of the manuscript.

\bibliographystyle{abbrv}
\nocite{*}
\bibliography{EKR06_refs}
	
\end{document}